\documentclass[10pt,english]{amsart}

\usepackage{amsopn}                    
\usepackage{amsmath} 
\usepackage{amssymb}
\usepackage{stmaryrd}
\usepackage[all,2cell,emtex]{xy}
\usepackage[mathscr]{euscript}
\usepackage{mathrsfs}
\usepackage{bbm}
\DeclareFontFamily{OT1}{pzc}{}
\DeclareFontShape{OT1}{pzc}{m}{it}%
             {<-> s * [1,150] pzcmi7t}{}
\DeclareMathAlphabet{\mathpzc}{OT1}{pzc}%
                                 {m}{it}

\theoremstyle{plain} 
\newtheorem{thm}{Theorem}[section]
\newtheorem{prop}[thm]{Proposition}
\newtheorem{lemma}[thm]{Lemma}
\newtheorem{cor}[thm]{Corollary}

\theoremstyle{remark}
\newtheorem{rem}[thm]{Remark}

\newtheorem{sch}[thm]{Scholium}
\theoremstyle{definition}

\newtheorem{paragr}[thm]{}

\theoremstyle{plain} 
\swapnumbers

\numberwithin{equation}{thm}

\renewcommand{\mathcal}{\mathpzc}
\renewcommand{\mathbb}{\mathbbm}

\renewcommand{\leq}{\leqslant}

\newcommand{\pref}[1]{[\op{#1},\mathit{sSet}]}

\newcommand{\To}{\longrightarrow}

\newcommand{\cat}{\mathpzc{Cat}}

\newcommand{\Hom}{\operatorname{\mathrm{Hom}}}
\newcommand{\sHom}{\operatorname{\mathpzc{Hom}}}

\newcommand{\op}[1]{{#1}^{\mathit{op}}}

\newcommand{\ho}{\operatorname{\mathbf{Ho}}}

\newcommand{\V}{\mathcal{V}}
\newcommand{\W}{\mathcal{W}}

\newcommand{\derL}{\mathbf{L}}

\def\TO#1{\mathrel{\hbox to #1pt{\rightarrowfill}}}
\def\OT#1{\mathrel{\hbox to #1pt{\leftarrowfill}}}

\def\limind{\mathop{\oalign{\rm lim\cr
\hidewidth$\longrightarrow$\hidewidth\cr}}}%

\renewcommand{\To}{{\hskip -2.5pt\xymatrixcolsep{1.3pc}\xymatrix{\ar[r]&}\hskip -2.5pt}}
\renewcommand{\longmapsto}{{\hskip -2.5pt\xymatrixcolsep{1.3pc}\xymatrix{\ar@{|->}[r]&}\hskip -2.5pt}}
\renewcommand{\mapsto}{{\hskip -2.5pt\xymatrixcolsep{.9pc}\xymatrix{\ar@{|->}[r]&}\hskip -2.5pt}}

\renewcommand{\to}{{\hskip -2.5pt\xymatrixcolsep{1pc}\xymatrix{\ar[r]&}\hskip -2.5pt}}
\newcommand{\todouble}{\xymatrixcolsep{1pc}\xymatrix{\ar@<.5ex>[r]\ar@<-.5ex>[r]&}}
\newcommand{\todoubleop}{\xymatrixcolsep{1pc}\xymatrix{\ar@<.5ex>[r]&\ar@<.5ex>[l]}}

\renewcommand{\hookrightarrow}{{\hskip -1.5pt\raise 1.5pt\vbox{\xymatrixcolsep{.9pc}\xymatrix{\ar@{^{(}->}[r]&}}\hskip -3.5pt}}

\newcommand{\intcoin}[5]{\raise 12pt\vbox{\xymatrixcolsep{.9pc}\xymatrixrowsep{.7pc}\xymatrix{%
\scriptstyle #1\ar[r]^{\scriptscriptstyle #5}\ar[d]_{\scriptscriptstyle #4}&\scriptstyle #3\\\scriptstyle #2}}}

\newcommand{\lc}{\mathbf{LC}}
\newcommand{\lcc}{\mathbb{LC}}
\newcommand{\Ho}{\mathbb{Ho}}
\newcommand{\der}{\mathbb{D}}
\newcommand{\derR}{\mathbf{R}}
\newcommand{\derhom}{\mathbb{Hom}}
\title{Locally constant functors}

\date{March 2008}

\author{Denis-Charles Cisinski}
\address{LAGA\\
CNRS~(UMR 7539)\\
Institut~Galil\'ee\\
Universit\'e Paris~13\\
Avenue \hbox{Jean-Baptiste} \hbox{Cl\'ement}\\
93430~Villetaneuse\\France}
\email{cisinski@math.univ-paris13.fr}
\urladdr{http://www.math.univ-paris13.fr/~cisinski/}
\dedicatory{To Michael Batanin, for his nice questions}

\begin{document}
\begin{abstract}
We study locally constant coefficients.
We first study the theory of homotopy Kan
extensions with locally constant coefficients in
model categories, and explain how it characterizes the
homotopy theory of small categories.
We explain how to interpret this in terms of
left Bousfield localization of categories of diagrams
with values in a combinatorial model category.
At last, we explain how the theory of homotopy
Kan extensions in derivators can be used to
understand locally constant functors.
\end{abstract}
\maketitle
\tableofcontents

\section{Homology with locally constant coefficients}
\begin{paragr}
Given a model category\footnote{We mean a Quillen
model category. However, we could
take any kind of model category giving rise
to a good theory of homotopy colimits (i.e. to
a Grothendieck derivator); see the work
of Andrei R\u{a}dulescu-Banu~\cite{RadBan}
for more general examples.} $\V$ with small colimits,
and a small category $A$,
we will write $[A,\V]$ for the category of functors from $A$ to $\V$.
Weak equivalences in $[A,\V]$
are the termwise weak equivalences. We denote
by $\ho([A,\V])$ the localization of $[A,\V]$
by the class of weak equivalences.
\end{paragr}

\begin{paragr}
We denote by $\lc(A,\V)$ the full subcategory of
the category $\ho([A,\V])$ whose objects are the
locally constant functors, i.e. the functors
$$F:A\To\V$$
such that for any map $a\To a'$ in $A$, the map
$$F_a\To F_{a'}$$
is a weak equivalence in $\V$, or equivalently, an isomorphism
in $\ho(\V)$ (where $F_a$ is
the evaluation of $F$ at the object $a$).

Note that for any functor $u:A\To B$, the inverse image
functor
\begin{equation}\label{definvim1}
u^* : [B,\V]\To [A,\V]
\end{equation}
preserves weak equivalences, so that it induces a functor
\begin{equation}\label{definvim2}
u^*:\ho([B,\V])\To\ho([A,\V])
\end{equation}
The functor $u^*$ obviously preserves locally
constant functors, so that
it induces a functor
\begin{equation}\label{definvim3}
u^* : \lc(B,\V)\To\lc(A,\V)\ .
\end{equation}
\end{paragr}

\begin{thm}\label{weakequivloccstcoh}
Let $u:A\To B$ be a functor whose nerve is a simplicial weak equivalence.
Then the functor \eqref{definvim3} is an equivalence of categories.
\end{thm}

\begin{paragr}
The proof of Theorem \ref{weakequivloccstcoh} will need a little
preparation.

Define a functor between small categories $u:A\To B$
to be a \emph{weak equivalence} if the functor \eqref{definvim3}
is an equivalence of categories.

A small category $A$ is \emph{aspherical} if the map $A\To e$
is a weak equivalence (where $e$ denotes the terminal category).
A functor $u:A\To B$ will be said to be
\emph{aspherical} if for any object $b$ of $B$, the
functor $A/b\To B/b$ is a weak equivalence
(where $A/b=(u\downarrow b)$ is the comma category of
objects of $A$ over $b$, and $B/b=(1_B\downarrow b)$).

Theorem \ref{weakequivloccstcoh} asserts that any functor whose nerve is a simplicial
weak equivalence is a weak equivalence in the sense define above.
This will follow from the following result.
\end{paragr}

\begin{thm}\label{minimal}
Let $\W$ be a class of functors between small categories.
We assume that $\W$ is a weak basic localizer in the sense of
Grothendieck~\cite{Mal}, which means that the following properties are satified.
\begin{itemize}
\item[La] Any identity is in $\W$. The class $\W$ satisfies the two out of three
property. If a map $i:A\To B$ has a retraction $r:B\To A$ such that $ir$ is in $\W$,
then $i$ is in $\W$.
\item[Lb] If a small category $A$ has a terminal object,
then the map from $A$
to the terminal category is in $\W$.
\item[Lc] Given a functor $u:A\To B$, if for any object $b$ of $B$,
the induced functor $A/b\To B/b$ is in $\W$, then $u$ is in $\W$.
\end{itemize}
Then any functor between small categories whose nerve
is a simplicial weak equivalence is in $\W$.
\end{thm}

\begin{proof}
See \cite[Theorem 6.1.18]{Ci3}.
\end{proof}

\begin{paragr}
To prove Theorem \ref{weakequivloccstcoh}, we will prove that
the class of weak equivalences satisfies the properties
listed in the previous theorem. Property La is easy to check.
It thus remains to prove properties Lb and Lc. In other
words, we have to prove that any category with a terminal
object is aspherical, and that any aspherical functor
is a weak equivalence.
We will use the theory of homotopy Kan extensions in $\V$.

Given a functor $u:A\To B$, the functor \eqref{definvim1}
has a left adjoint
\begin{equation}\label{definvim5}
u_! : [A,\V]\To [B,\V]
\end{equation}
which admits a total left derived functor
\begin{equation}\label{definvim6}
\derL u_! : \ho([A,\V])\To\ho([B,\V])
\end{equation}
The functor $\derL u_!$ is also a left adjoint
of the functor \eqref{definvim2}; see \cite{Ci1,RadBan}.
When $B$ is the terminal category, we will write \smash{$\derL\limind_A=\derL u_!$} for the
corresponding homotopy colimit functor.
\end{paragr}

\begin{paragr}
For each object $b$ of $B$, we have the following pullback square
of categories
\begin{equation}\label{definvim7}\begin{split}\xymatrix{
A/b\ar[d]_{u/b}\ar[r]^v&A\ar[d]^u\\
B/b\ar[r]_w&B
}\end{split}\end{equation}
(where $w$ is the obvious forgetful functor).
Given a functor $F$ from $A$ (resp. $B$) to $\V$, we will write
\begin{equation}\label{definvim8}
F/b=v^*(F)\quad\text{(resp.}\ F/b=w^*(F)\ ).
\end{equation}
This gives the following formula for a functor $F:B\To\V$
\begin{equation}\label{definvim9}
u^*(F)/b=(u/b)^*(F/b)\ .
\end{equation}
It is a fact that left homotopy Kan extensions
can be computed pointwise (like in ordinary category theory),
which can be formulated like this:
\end{paragr}

\begin{prop}\label{changebase}
For any functor $F:A\To\V$, and any object $b$ of $B$, the
base change map
$$\derL\limind_{A/b}F/b\To \derL u_!(F)_b$$
is an isomorphism in $\ho(\V)$.
\end{prop}

\begin{proof}
See \cite{Ci1} or, in a more general context, \cite[Theorem 9.6.5]{RadBan}.
\end{proof}

\begin{prop}\label{isolocal}
Let $I$ be a small category. A morphism $F\To G$
in $\ho([I,\V])$ is an isomorphism if and only for
any object $i$ of $I$, the map $F_i\To G_i$
is an isomorphism in $\ho(\V)$.
\end{prop}

\begin{proof}
See \cite{Ci1} or \cite[Theorem 9.7.1]{RadBan}.
\end{proof}

\begin{prop}\label{loccsthomology}
Let $u:A\To B$ be a weak equivalence of small categories.
Then, for any locally constant functor $F:B\To \V$, the
map
$$\derL\limind_A u^*(F)\To\derL\limind_B F$$
is an isomorphism in $\ho(\V)$.
\end{prop}

\begin{proof}
Given a small category $I$ and an object $X$, denote by $X_I$
the constant functor from $I$ to $\V$ with value $X$.
Let $F:B\To \V$ be a locally constant functor.
Using the fact that \eqref{definvim3} is fully faithful,
we see that
$$\Hom_{\ho([B,\V])}(F,X_B)\To\Hom_{\ho([A,\V])}(u^*(F),u^*(X_B))$$
is bijective. As $u^*(X_B)=X_A$, the identifications
$$\begin{aligned}
\Hom_{\ho([B,\V])}(F,X_B)&\simeq\Hom_{\ho(\V)}(\derL\limind_B F,X)\\
\Hom_{\ho([A,\V])}(u^*(F),X_A)&\simeq\Hom_{\ho(\V)}(\derL\limind_A u^*(F),X)
\end{aligned}$$
and the Yoneda Lemma applied to $\ho(\V)$ ends the proof.
\end{proof}

\begin{cor}\label{loctriv}
Let $I$ be an aspherical category, and $F:I\To\V$
be a locally constant functor. Then for any object $i$
of $I$, the map
$$F_i\To\derL\limind_I F$$
is an isomorphism in $\ho(\V)$.
\end{cor}

\begin{proof}
Apply Proposition \ref{loccsthomology} to the functor from
the terminal category to $I$ defined by $i$.
\end{proof}

\begin{prop}\label{termobjectasph}
Any small category which has a terminal object is aspherical.
\end{prop}

\begin{proof}
Let $I$ be a small category with a terminal object $\omega$.
This means that the functor $p$ from $I$ to the terminal category has a
right adjoint (which is has to be fully faithful).
But
$$A\longmapsto \ho([A,\V])$$
obviously defines a $2$-functor, which implies that the functor
$$p^*:\ho(\V)\To\ho([I,\V])$$
is fully faithful and that the evaluation functor at $\omega$
is a left adjoint to $p^*$. In particular, for any functor $F$
from $I$ to $\V$, we have
$$\derL\limind_I F=F_\omega\, .$$
If moreover $F$ is locally constant, then it follows from
Proposition \ref{isolocal} that the co-unit map
$$F\To p^*F_\omega=p^*\derL\limind_I F$$
is an isomorphism. This shows that $p^*$ is also
essentially surjective and ends the proof.
\end{proof}

\begin{cor}\label{aspheraspher}
A functor $u:A\To B$ is aspherical if and only if for any
object $b$ of $B$, the category $A/b$ is aspherical.
\end{cor}

\begin{proof}
As the class of weak equivalences satifies the two out of three
property, this follows from the fact that the category $B/b$
has a terminal object (namely $(b,1_b)$).
\end{proof}

\begin{paragr}
A functor $u:A\To B$ is \emph{locally constant} if for any
map $b\To b'$ in $B$, the functor $A/b\To A/b'$ is
a weak equivalence. For example, any aspherical functor
is locally constant.
\end{paragr}

\begin{prop}[Formal Serre spectral sequence]\label{serre}
If $u:A\To B$ is locally constant, then the functor \eqref{definvim6}
preserves locally constant functors. In particular, it induces a functor
$$\derL u_! : \lc(A,\V)\To\lc(B,\V)$$
which is a left adjoint to the functor \eqref{definvim3}.
\end{prop}

\begin{proof}
Let $F$ be a locally constant functor, and $\beta : b\To b'$
be a map in $B$, We have to show that the induced map
$$\derL u_! (F)_b\To\derL u_!(F)_{b'}$$
is an isomorphism in $\ho(\V)$.
Denote by $j^{}_\beta : A/b\To A/b'$ the functor induced
by $\beta$ (which is a weak equivalence by assumption on $u$).
With the notations \eqref{definvim8}, we have
$$j^*_{\beta}(F/b')=F/b\, .$$
This corollary thus follows immediately from
Proposition \ref{changebase} and from Proposition \ref{loccsthomology}.
\end{proof}

\begin{prop}[Formal Quillen Theorem A]\label{main}
Any aspherical functor is a weak equivalence.
\end{prop}

\begin{proof}
Let $u:A\To B$ be an aspherical functor.
Let $F: B\to \V$ be a locally constant functor.
We will prove that the co-unit map
$$\derL u_! u^*(F)\To F$$
is an isomorphism. According to Proposition \ref{isolocal},
it is sufficient to prove that for any object $b$ of $B$, the map
$$(\derL u_! u^*(F))_b\To F_b$$
is an isomorphism in $\ho(\V)$.
This follows immediately from the computations below.
$$\begin{aligned}
(\derL u_!u^*(F))_b&\simeq \derL\limind_{A/b} u^*(F)/b\quad\text{(Proposition \ref{changebase})}\\
&=\derL\limind_{A/b}(u/b)^*(F/b)\quad\text{(Formula \eqref{definvim9})}\\
&\simeq\derL\limind_{B/b} F/b\quad\text{(Proposition \ref{loccsthomology} applied to $u/b$)}\\
&\simeq F_b\quad\text{(because $(b,1_b)$ is a terminal object of $B/b$).}
\end{aligned}$$
Consider now a locally constant functor $F:A\To \V$. We will show that the
unit map
$$F\To u^*\derL u_!(F)$$
is an isomorphism. By virtue of Proposition \ref{isolocal},
we are reduced to prove that, for any object $a$ of $A$, the map
$$F_a\To (u^*\derL u_!(F))_a$$
is an isomorphism. We compute again
$$\begin{aligned}
\quad F_a&\simeq\derL\limind_{A/u(a)} F/u(a)
\quad\text{(Corollary \ref{loctriv} for $I=A/u(a)$ and $i=(a,1_{u(a)})$)}\\
&\simeq\derL u_!(F)_{u(a)}\quad\text{(Proposition \ref{changebase})}\\
&=(u^*\derL u_!(F))_a
\end{aligned}$$
and this ends the proof.
\end{proof}

\begin{proof}[Proof of Theorem \ref{weakequivloccstcoh}]
It is sufficient to check that the class of weak equivalences
satisfy the properties listed in Theorem \ref{minimal}.
The class of weak equivalences obviously satifies property La.
Property Lb follows from Proposition \ref{termobjectasph}, and property Lc
from Proposition \ref{main}.
\end{proof}

\begin{cor}\label{corcarweakequivcat}
Let $u:A\To B$ a functor between small categories. The nerve
of $u$ is a simplicial weak equivalence if and only if for
any model category $\V$, the functor
$$u^*:\lc(B,\V)\To\lc(A,\V)$$
is an equivalence of categories.
\end{cor}

\begin{proof}
Theorem \ref{weakequivloccstcoh} asserts this is a necessary condition.
It is very easy to check that this is also sufficient: we can either
use Proposition \ref{loccsthomology} and \cite[6.5.11]{Ci3}, or
we can use the fact that, for a given small category $A$, the
homotopy colimit of the constant functor indexed by $A$
whose value is the terminal simplicial set is
precisely the nerve of $A$ (which is completely obvious
if we consider for example the Bousfield-Kan construction of
homotopy colimits).
\end{proof}

\section{Model structures for locally constant functors}

\begin{paragr}\label{combmodecat}
We consider now a left proper combinatorial model
category $\V$ and a small category $A$
(see \cite{beke1} for the definition of
a combinatorial model category). The category
of functors $[A,\V]$ has two canonical model structures.
The \emph{projective model structure} on $[A,\V]$
is defined as follows: the weak equivalences are the
termwise weak equivalences, and the fibrations are the termwise
fibrations. The \emph{injective model structure} on $[A,\V]$
is defined dually: the weak equivalences are the
termwise weak equivalences, and the cofibrations are the
termwise cofibrations. One can check that the identity
functor is a left Quillen equivalence from
the projective model structure to the injective
model structure (this is just an abstract way to say that
all the cofibrations of the projective model
structure are termwise cofibrations, which is easy to check;
see for example \cite[Lemma 3.1.12]{Ci3}).
These two model structures are left proper.
\end{paragr}

\begin{paragr}
We fix a (regular) cardinal $\alpha$ with the following properties
(see \cite{dug2}).
\begin{itemize}
\item[(a)] Any object of $\V$ is a $\alpha$-filtered colimit
of $\alpha$-small objects.
\item[(b)] The class of weak equivalences of $\V$
is stable by $\alpha$-filtered colimits.
\item[(c)] There exists a cofibrant resolution
functor $Q$ which preserves $\alpha$-filtered colimits. 
\end{itemize}
Given an object $a$ of $A$, we denote by
$$a_!: \V\To [A,\V]$$
the left adjoint to the evaluation functor at $a$.
We define $S$ as the (essentially small) set of maps
of shape
\begin{equation}
a_!(QX)\To a^\prime_!(QX)
\end{equation}
associated to each map $a\To a'$ in $A$ and each
$\alpha$-accessible object $X$ (and $Q$ is some
fixed cofibrant resolution functor satisfying
the condition (c) above).

We define the \emph{projective local model structure}
on $[A,\V]$ as the left Bousfield localization of the
projective model structure on $[A,\V]$ by $S$.
The \emph{injective local model structure} on $[A,\V]$
is the left Bousfield localization of the injective
model structure on $[A,\V]$ by $S$. It is clear that the identity
functor is still a left Quillen equivalence from
the projective local model structure to the injective local
model structure. The weak equivalences of these two model
structures will be called the \emph{local weak equivalences}.
\end{paragr}

\begin{prop}\label{carloccstfibrant}
A functor $F:A\To\V$ is fibrant in the projective
(resp. injective) local model structure if and only
if it is fibrant for the projective (resp. injective)
model structure and if it is locally constant.
\end{prop}

\begin{proof}
It is sufficient to prove this for the projective local model
structure. Note first that, thanks to condition (b),
for any $\alpha$-filtered category $I$ and any functor $F$ from
$I$ to $[A,\V]$, the natural map
$$\derL\limind_I F\To\limind_I F$$
is an isomorphism in $\ho([A,\V])$. Hence it remains an
isomorphism in the homotopy category of the projective local model
structure. This implies that local weak equivalences are stable
by $\alpha$-filtered colimits.
Conditions (a) and (b) thus imply that for any object
$X$ of $\V$, and any arrow $a\To a'$ in $A$, the map
$$a_!(QX)\To a^\prime_!(QX)$$
is a local weak equivalence.
If $a$ is an object of $A$, $X$ is an object
of $\V$, and $F$ a functor from $A$ to $\V$, then
$$\Hom_{\ho(\V)}(X,F_a)\simeq\Hom_{\ho([A,\V])}(\derL a_!(X),F)
\simeq\Hom_{\ho([A,\V])}(a_!(QX),F).$$
It is now easy to see that, if moreover $F$ is fibrant
for the projective model structure, then it is fibrant
for the projective local model structure if and only
if it is locally constant.
\end{proof}

\begin{cor}\label{lcho}
The localization of $[A,\V]$ by the
class of local weak equivalences is $\lc(A,\V)$.
\end{cor}

\begin{cor}
The inclusion functor $\lc(A,\V)\To\ho([A,\V])$
has a left adjoint.
\end{cor}

\begin{prop}\label{functquillenlc}
Let $u:A\To B$ be a functor between small categories.
Then the functor
$$u^*:[B,\V]\To[A,\V]$$
is a right Quillen functor for
the projective local model structures.

If moreover the nerve of $u$ is a simplicial weak equivalence, then
the functor $u^*$ is a right Quillen equivalence.
\end{prop}

\begin{proof}
The left adjoint $u_!$ of $u^*$ preserves cofibrations:
this is obviously a left Quillen functor for the
projective model structures.
It is thus sufficient to check that $u^*$
preserves fibrations between fibrant objects;
see \cite[Proposition 7.15]{joytier4}.
It follows from Proposition \ref{carloccstfibrant}
that fibrations between fibrant objects are just
fibrations of the projective model structure
between fibrant objects of the projective model
structure which are locally constant. It is clear that
$u^*$ preserves this property. This proves that $u^*$
is a right Quillen functor. The last assertion
follows from Theorem \ref{weakequivloccstcoh}.
\end{proof}

\begin{rem}\label{serreencore}
According to the preceding proposition,
the functor $u_!$ has a total left derived functor
$$\derL u_! :\lc(A,\V)\To\lc(B,\V)\, .$$
It also has a total left derived functor
$$\derL u_!: \ho([A,\V])\To \ho([A,\V])\, .$$
but, in general, the diagram (in which $i_A$ and $i_B$ denote the inclusion functors)
$$\xymatrix{
\lc(A,\V)\ar[r]^{\derL u_!}\ar[d]_{i_A}&\lc(B,\V)\ar[d]^{i_B}\\
\ho([A,\V])\ar[r]_{\derL u_!}&\ho([B,\V])
}$$
does not (even essentially) commute. There is only a natural map
$$\derL u_!i_A(F)\To i_B\derL u_!(F)\, .$$
However, Proposition \ref{serre} asserts that this natural map is an isomorphism
whenever $u$ is locally constant.
\end{rem}

\begin{prop}\label{functquillenlc2}
Let $u:A\To B$ be a functor between small categories.
Assume that the functor $\op{u}:\op{A}\To\op{B}$
is locally constant.
Then the functor
$$u^*:[B,\V]\To[A,\V]$$
is a left Quillen functor for
the injective local model structures.

If moreover the nerve of $u$ is a simplicial weak equivalence, then
the functor $u^*$ is a left Quillen equivalence.
\end{prop}

\begin{proof}
We know that $u^*$
is a left Quillen functor for the injective model
structure. Hence, by virtue of Proposition \ref{carloccstfibrant},
it is sufficient to prove that the total right derived functor
$$\derR u_*:\ho([A,\V])\To \ho([B,\V])$$
preserves locally constant functors. But this latter property
is just Proposition \ref{serre} applied to $\op{\V}$.
The last assertion
follows again from Theorem \ref{weakequivloccstcoh}.
\end{proof}

\section{Locally constant coefficients in Grothendieck derivators}

\begin{paragr}
We start this section by fixing some notations.

Given a cocomplete model category $\V$, we denote by $\Ho(\V)$
the derivator associated to $\V$; see \cite[Theorem 6.11]{Ci1}.

Let $A$ be a small category.
We will consider the category $[\op{A},\mathit{sSet}]$ of simplicial
presheaves on $A$ endowed with the projective model
structure.
Given a subcategory $S$ of $A$, we denote by $L_S\pref{A}$
the left Bousfield localization of the projective model
structure on $[\op{A},\mathit{sSet}]$ by $S$
(where $S$ is seen as a set of maps in $[\op{A},\mathit{sSet}]$
via the Yoneda embedding). The fibrant objects of
$L_S\pref{A}$ are the simplicial presheaves $F$ on $A$
which are termwise Kan complexes and which sends the maps of $S$
to simplicial homotopy equivalences. In particular, in case $A=S$,
$L_A\pref{A}$ is the projective local model structure on
$[\op{A},\mathit{sSet}]$ studied in the previous section.
\end{paragr}

\begin{paragr}\label{definternhomder}
Given two prederivators $\der$ and $\der'$, we denote by
$\sHom(\der,\der')$ the category of morphisms of
derivators; see \cite{Ci1}. If $\der$ and $\der'$
are derivators, we denote by $\sHom_!(\der,\der')$
the full subcategory of $\sHom(\der,\der')$
whose objects are the morphisms of prederivators
which preserves left homotopy Kan extensions (which
are called cocontinuous morphisms in \cite{Ci1}).

Given a (small) category, we denote by
$\underline{A}$ the prederivator which associates to
each small category $I$ the category $[\op{I},A]$
of presheaves on $I$ with values in $A$. This defines
a $2$-functor from the $2$-category of small categories
to the $2$-category of prederivators.
Note that we have a Yoneda Lemma for prederivators:
given a small category $A$ and a prederivator $\der$,
the functor
\begin{equation}\label{yonedapreder}
\sHom(\underline{A},\der)\To\der(\op{A})
\quad ,  \qquad F\longmapsto F(1_A)
\end{equation}
is an equivalence of categories.
\end{paragr}

\begin{thm}\label{univ}
For any derivator $\der$,
the composition by the Yoneda embedding $h:\underline{A}\To\Ho([\op{A},\mathit{sSet}])$
induces an equivalence of categories
$$\sHom_!(\Ho([\op{A},\mathit{sSet}]),\der)\simeq\sHom(\underline{A},\der)\, .$$
\end{thm}

\begin{proof}
This is a translation of \cite[Corollary 3.26]{Ci4} using \eqref{yonedapreder}.
\end{proof}

\begin{paragr}
We denote by $\sHom_S(\underline{A},\der)$
the full subcategory of morphisms $\underline{A}\To\der$
such that the induced functor $A\To\der(e)$
sends the maps of $S$ to isomorphims (where $e$
denotes the terminal category).
A formal consequence of Theorem \ref{univ} is:
\end{paragr}

\begin{thm}\label{univ2}
For any derivator $\der$,
the composition by the Yoneda morphism $h:\underline{A}\To\Ho(L_S\pref{A})$
induces an equivalence of categories
$$\sHom_!(\Ho(L_S\pref{A}),\der)\simeq\sHom_S(\underline{A},\der)\, .$$
\end{thm}

\begin{proof}
This follows immediately from Theorem \ref{univ}
and from the universal property of left Bousfield localization
for derivators; see \cite[Theorem 5.4]{tabuada}.
\end{proof}

\begin{paragr}
Given a small category $A$ and a derivator $\der$, we define
\begin{equation}
\lc(A,\der)=\sHom_A(\underline{A},\der)\, .
\end{equation}
It is clear that for a model category $\V$, we have by definition
\begin{equation}
\lc(A,\V)=\lc(A,\Ho(\V))\, .
\end{equation}
\end{paragr}

\begin{cor}
Let $u:A\To B$ be a functor between small categories. Then
the nerve of $u$ is a simplicial weak equivalence if and only if for
any derivator $\der$, the functor
$$u^*:\lc(B,\der)\To\lc(A,\der)$$
is an equivalence of categories.
\end{cor}

\begin{proof}
As any model category gives rise to a derivator, this is certainly
a sufficient condition, by virtue of Corollary \ref{corcarweakequivcat}.
It thus remains to prove that this is a necessary condition.
The nerve of the functor $u$ is a simplicial weak equivalence
if and only if the nerve of $\op{u}:\op{A}\To\op{B}$
is so. This result is thus a consequence of Proposition \ref{functquillenlc},
of Theorem \ref{univ2}, and of the fact that any Quillen equivalence
induces an equivalence of derivators.
\end{proof}

\begin{lemma}\label{inclusionlccocontinue}
Let $A$ be a small category.
The inclusion morphism
$$i:\Ho(L_A\pref{A})\To\Ho([\op{A},\mathit{sSet}])$$
(defined as the right adjoint of the localization morphism)
preserves left homotopy Kan extensions.
\end{lemma}

\begin{proof}
It is sufficient to check that it preserves
homotopy colimits; see \cite[Proposition 2.6]{Ci4}.
This reduces to check that locally constant
functors are stable by homotopy colimits
in the model category of simplicial presheaves
on a small category, which is obvious.
\end{proof}

\begin{prop}\label{leftrightlc}
For any derivator $\der$ and any small category $A$,
the inclusion functor
$$\lc(A,\der)\To\sHom(\underline{A},\der)$$
has a left adjoint and a right adjoint.
\end{prop}

\begin{proof}
We have a localization morphism
$$\gamma:\Ho([\op{A},\mathit{sSet}])\To\Ho(L_A\pref{A})$$
which has a right adjoint in the $2$-category of prederivators
$$i:\Ho(L_A\pref{A})\To\Ho([\op{A},\mathit{sSet}])\, .$$
We know that $\gamma$ is cocontinuous (as it
comes from a left Quillen functor; see \cite[Proposition 6.2]{Ci1}).
The previous lemma asserts that $i$ is cocontinuous
as well. It thus follows from the fact $\sHom_!(-,\der)$
is $2$-functor and from Theorem \ref{univ2} that
the inclusion functor
$$\lc(A,\der)\To\sHom(\underline{A},\der)$$
(which is induced by $\gamma$) has a left adjoint
(which is induced by $i$). Applying this to
the opposite derivator $\op{\der}$ (and replacing $A$ by
$\op{A}$) also gives a right adjoint.
\end{proof}

\begin{cor}
Let $u:A\To B$ be a functor between small categories.
For any derivator $\der$, the inverse image functor
$$u^*:\lc(B,\der)\To\lc(A,\der)$$
has a left adjoint and a right adjoint.
\end{cor}

\begin{paragr}
It is possible to construct a prederivator $\lcc(A,\der)$
such that
\begin{equation}
\lcc(A,\der)(e)=\lc(A,\der)
\end{equation}
(where $e$ still denotes the terminal category).
If $\der$ is a derivator, and $A$ is a small category,
then we define a derivator $\der^A$ by the formula
\begin{equation}
\der^A(I)=\der(\op{A}\times I)\, .
\end{equation}
It is easy to see that $\der^A$ is again a derivator.
Moreover, the homotopy colimits in $\der^A$
can be computed termwise; see \cite[Proposition 2.8]{Ci4}.
In the case where $\der=\Ho(\V)$ for a model category $\V$,
we get the formula
\begin{equation}
\Ho(\V)^A(I)=\ho([A\times \op{I},\V])\, .
\end{equation}
The prederivator $\lcc(A,\der)$ is the full subprederivator
of $\der^A$ defined by the formula
\begin{equation}
\lcc(A,\der)(I)= \lc(A,\der^{\op{I}})\, .
\end{equation}
In other words, $\lcc(A,\der)(I)$ is the full subcategory of
$\der(\op{A}\times I)$ whose objects are the objects $F$
of $\der(\op{A}\times I)$ such that the induced functor
$$\mathsf{dia}(F):A\To [\op{I},\der(e)]$$
sends any morphism of $A$ to isomorphisms.
\end{paragr}

\begin{thm}\label{leftrightlcder}
For any small category $A$, and any derivator $\der$,
the prederivator $\lcc(A,\der)$ is a derivator, and
the full inclusion
$$\lcc(A,\der)\To\der^A$$
has a left adjoint and a right adjoint.
\end{thm}

\begin{proof}
The proof will follow essentially the same lines
as the proof of Proposition \ref{leftrightlc}.

Recall that there is an internal Hom for prederivators:
if $\der$ and $\der'$ are prederivators, we define
a prederivator $\derhom(\der,\der')$ by the formula
$$\derhom(\der,\der')(I)=\sHom(\der,{\der^{\prime}}^{\op{I}})$$
for any small category $I$; see \cite[Corollary 5.3]{Ci4}.
If moreover $\der$ and $\der'$ are derivators, we define
a prederivator $\derhom_!(\der,\der')$ as a full subprederivator
of $\derhom(\der,\der')$ as follows: for each small category $I$,
we put
$$\derhom_!(\der,\der')(I)=\sHom_!(\der,{\der^{\prime}}^{\op{I}})\, .$$
Then $\derhom_!(\der,\der')$ is again a derivator; see \cite[Proposition 5.8]{Ci4}.
Theorem \ref{univ} gives the following result.
If $A$ is a small category, then
for any derivator $\der$, the Yoneda map
$h:\underline{A}\To\Ho(\pref{A})$
induces an equivalence of derivators
$$\derhom_!(\Ho(\pref{A}),\der)\simeq\derhom(\underline{A},\der)=\der^A\, .$$
Similarly, Theorem \ref{univ2} implies that the Yoneda
map $h:\underline{A}\To\Ho(L_A\pref{A})$
induces an equivalence of derivators
$$\derhom_!(\Ho(L_A\pref{A}),\der)\simeq\lcc(A,\der)\, .$$
Thanks to Lemma \ref{inclusionlccocontinue} and to the fact
that $\derhom_!(-,\der)$ is a $2$-functor, the adjunction
$$\gamma:\Ho(\pref{A})\rightleftarrows\Ho(L_A\pref{A}):i$$
thus induces an adjunction
$$i^*:\derhom_!(\Ho(\pref{A}),\der)\rightleftarrows\derhom_!(\Ho(L_A\pref{A}),\der):\gamma^*\, .$$
In particular, we see that $\lcc(A,\der)$ is a derivator (as it is equivalent
to the derivator $\derhom_!(\Ho(L_A\pref{A}),\der)$), and
we get an adjunction of derivators
$$\der^A\rightleftarrows\lcc(A,\der)\, .$$
Applying this to the opposite derivator $\op{\der}$ gives the other adjoint.
\end{proof}

\begin{rem}
The preceding result can be interpreted as follows in terms
of model categories. Consider a small category $A$ and a complete
and cocomplete model category $\V$. Then $\Ho(\V)$
is a derivator, so that $\lcc(A,\Ho(\V))$ is a derivator as well.
Denote by $\mathit{LC}(A,\V)$ the full subcategory of
$[A,\V]$ whose objects are the locally constant functors.
One can then verify that the prederivator associated to
the category $\mathit{LC}(A,\V)$ (by inverting the termwise weak equivalences)
is canonically equivalent to $\lcc(A,\Ho(\V))$; this can be expressed by the formula
$$\Ho(\mathit{LC}(A,\V))\simeq\lcc(A,\Ho(\V))\, .$$
This means that the left Bousfield localizations discussed in \ref{combmodecat}
for combinatorial model categories always exist in the setting
of derivators. Theorem \ref{leftrightlcder} implies that
such Bousfield localizations actually exist in the setting
of ABC~cofibration categories developped in \cite{RadBan}.
\end{rem}

\section{Galois correspondence and homotopy distributors}

\begin{paragr}
Let $A$ and $B$ be small categories. We get from Theorem \ref{univ2}
the following canonical equivalence of categories
\begin{equation}\label{galois1}\begin{split}\begin{aligned}
\sHom_!(\Ho(L_B\pref{B}),\Ho(L_A\pref{A}))&\simeq\sHom_B(\underline{B},\Ho(L_A\pref{A}))\\
&\simeq\ho(L_{A\times \op{B}}[\op{A}\times B,\mathit{sSet}])\, .
\end{aligned}\end{split}\end{equation}
Moreover, we have an equivalence of categories
\begin{equation}\label{galois2}
\ho(L_{A\times \op{B}}[\op{A}\times B,\mathit{sSet}])\simeq\ho(\cat/ A\times B)
\end{equation}
where $\ho(\cat/ A\times B)$ denotes the localization
of the category of small categories over $A\times B$ by the class of functors
(over $A\times B$) whose nerve are simplicial weak equivalences;
this follows for example from \cite[Corollaries 4.4.20 and 6.4.27]{Ci3}
and from the fact $B$ and $\op{B}$ have the same homotopy type.

The induced equivalence of categories
\begin{equation}\label{galois3}
S:\ho(\cat/ A\times B)\To\sHom_!(\Ho(L_A\pref{A}),\Ho(L_B\pref{B}))
\end{equation}
can be described very explicitely: its composition with the localization
functor from $\cat/ A\times B$ to $\ho(\cat/ A\times B)$ is the functor
\begin{equation}\label{galois5}
s:\cat/ A\times B\To\sHom_!(\Ho(L_A\pref{A}),\Ho(L_B\pref{B}))
\end{equation}
which can be described as follows. Consider a functor $C\To A\times B$.
It is determined by a pair of functors $p:C\To A$ and $q:C\To B$.
The functor $q$ induces an inverse image morphism
\begin{equation}\label{galois6}
q^* : L_B\pref{B}\To L_C\pref{C}
\end{equation}
which happens to be a right Quillen functor for
the projective local model structures; see Proposition \ref{functquillenlc}.
It thus defines a continuous morphism of derivators (see \cite[Proposition 6.12]{Ci1})
\begin{equation}\label{galois7}
\derR q^* : \Ho(L_B\pref{B})\To\Ho(L_C\pref{C})\ .
\end{equation}
Using the equivalences of type $\Ho(L_B\pref{B})\simeq\lcc(\op{B},\Ho(\mathit{sSet}))$,
we see that $\derR q^*$ corresponds to the restriction
to the derivators $\lcc(\op{B},\Ho(\mathit{sSet}))$
and $\lcc(\op{C},\Ho(\mathit{sSet}))$ of the
inverse image map $q^*:\der^{\op{B}}\To\der^{\op{C}}$
for $\der=\Ho(\mathit{sSet})$ (which
is cocontinuous, by virtue of \cite[Proposition 2.8]{Ci4}).
We thus conclude from Lemma \ref{inclusionlccocontinue} that
\eqref{galois7} is also cocontinuous. The functor $p$ induces a left Quillen
functor for the projective local model structures (by Proposition \ref{functquillenlc}
again)
\begin{equation}\label{galois8}
p_! : L_C\pref{C}\To L_A\pref{A}\ .
\end{equation}
This defines a cocontinuous morphism of derivators (by the dual
of \cite[Proposition 6.12]{Ci1})
\begin{equation}\label{galois9}
\derL p_! : \Ho(L_C\pref{C})\To\Ho(L_A\pref{A})\ .
\end{equation}
The functor \eqref{galois5} is simply defined by sending the pair $(p,q)$
to the composition of \eqref{galois7} and \eqref{galois9}.
\begin{equation}\label{galois10}
s(p,q)=\derL p_!\derR q^* : \Ho(L_B\pref{B})\To\Ho(L_A\pref{A})\ .
\end{equation}
\end{paragr}

\begin{prop}\label{galoisexactder}
Given a functor $(p,q):C\To A\times B$, the following conditions
are equivalent.
\begin{itemize}
\item[(a)] The morphism \eqref{galois10}
is continuous (i.e. preserves homotopy limits).
\item[(b)] The morphism \eqref{galois9} is continuous.
\item[(c)] The functor
$\derL p_!\derR q^*:\ho(L_B\pref{B})\To\ho(L_A\pref{A})$
preserves terminal objects.
\item[(d)] The functor
$\derL p_! : \ho(L_C\pref{C})\To\ho(L_A\pref{A})$
preserves terminal objects.
\item[(e)] The morphism \eqref{galois9} is an equivalence of
derivators.
\item[(f)] The functor
$\derL p_! : \ho(L_C\pref{C})\To\ho(L_A\pref{A})$
is an equivalence of categories.
\item[(g)] The nerve of $p$ is a simplicial weak equivalence.
\end{itemize}
\end{prop}

\begin{proof}
The functor \eqref{galois8} is a left Quillen equivalence
(for the projective local model structures) if and only
if for any small category $I$, the induced functor
$$p_! : [I,L_C\pref{C}]\To[I,L_A\pref{A}]$$
is a left Quillen equivalence. This proves that the conditions
(e) and (f) are equivalent. It is obvious that condition (e)
implies condition (b). The fact that condition (g) implies
condition (f) can be obtained,
for example, using Theorem \ref{weakequivloccstcoh}.
It is clear that condition (b) implies conditions (a) and (d), and
that conditions (a) or (d) implies condition (c). To finish the proof, we
will show that the condition (c) implies (g).

Under the equivalences of type $\ho(L_X\pref{X})\simeq\ho(\cat/X)$,
the functor $\derL p_!$ corresponds to the functor from $\ho(\cat/A)$
to $\ho(\cat/A)$ which is induced by composition with $p$.
Similarly, the functor $\derR q^*$
corresponds to the functor from $\ho(\cat/B)$ to $\ho(\cat/C)$
which sends a functor $X\To B$ to the projection $X\times^h_B C\To C$
(where $X\times^h_B C$ denotes the homotopy fiber product of $X$
and $C$ over $B$). These descriptions show immediately that
the condition (c) implies (g). This ends the proof.
\end{proof}

\begin{paragr}
We refer to \cite{Mal,Mal2,Ci3} for the notion of smooth functor
and of proper functor (with respect to the minimal basic localizer).
The first reason we are interested by this notion is that
these functors have very good properties with respect to
homotopy Kan extensions; see \cite[Section 3.2]{Mal}.
The second reason of our interest for this class of functors is
the following statement.
\end{paragr}

\begin{prop}\label{catbrown}
The category of small categories is endowed with a
structure of category of fibrant objects in the sense of Brown~\cite{Br},
for which the weak equivalences are the functors whose nerve
is a simplicial weak equivalence, and the fibrations are the
smooth and proper functors. Moreover, the factorizations
into a weak equivalence followed by a fibration can
be made functorially.
\end{prop}

\begin{proof}
Any functor to the terminal category is smooth and proper
(so that any small category will be fibrant).
Functors which are smooth and proper are stable
under base change and composition (see \cite[Corollary 3.2.4
and Proposition 3.2.10]{Mal}). It follows from
\cite[Corollaries 6.4.8 and 6.4.18]{Ci3}
and from \cite[Proposition 3.2.6]{Mal} that
the class of trivial fibrations (i.e. of
smooth and proper functors which are weak equivalences)
is stable by pullbacks.
By virtue of \cite[Proposition 6.4.14]{Ci3}, the pullback of
a weak equivalence by a smooth and proper functor
is a weak equivalence. To finish the proof, it is sufficient
to prove that any functor can factor (functorially) through a weak
equivalence followed by a smooth and proper functor, which
is a consequence of \cite[Theorem 5.3.14]{Ci3}.
\end{proof}

\begin{cor}\label{catbrowncor}
The localization of the full subcategory of
$\cat/A\times B$ whose objects are the
functors $(p,q):C\To A\times B$
such that $p$ and $q$ are smooth and proper
by the class of weak equivalences is canonicaly
equivalent to $\ho(\cat/A\times B)$.
\end{cor}

\begin{paragr}
The simplicial localization $L(\cat)$ of $\cat$ by
the class of weak equivalences can be described using
the structure of category of fibrant objects
given by Proposition \ref{catbrown}.
In particular, the simplicial set $\Hom_{L(\cat)}(A,B)$
can be described as the nerve of the category
$\mathit{Map}(A,B)$, which is defined as the
full subcategory of $\cat/A\times B$ whose
objects are the functors $(p,q):C\To A\times B$ such that $p$
is a trivial fibration (i.e. a functor which is
smooth, proper, and a weak equivalence).
It is easy to see from Proposition \ref{catbrown} that
the fundamental groupoid of $\mathit{Map}(A,B)$
is equivalent to the full subcategory of
$\ho(\cat/A\times B)$ whose objects are the
functors $(p,q):C\To A\times B$ such that $p$
is a weak equivalence. In other words, Proposition \ref{galoisexactder}
can now be reformulated as follows.
\end{paragr}

\begin{cor}[Galois reconstruction theorem]
The groupoid $\pi_1(\mathit{Map}(A,B))$ is canonically
equivalent to the category of cocontinuous morphisms of derivators
from $\Ho(L_B\pref{B})$
to $\Ho(L_A\pref{A})$ which preserve finite homotopy limits.
\end{cor}

\begin{paragr}
Let us explain why the preceding corollary can be interpreted
as a Galois reconstruction theorem.
Given a small category $A$, if we think
of $\Ho(L_A\pref{A})$ as the ``topos of representations
of the $\infty$-groupoid associated to $A$'', it is natural
to define the functor of points of $\Ho(L_A\pref{A})$
by
$$B\longmapsto \sHom_!^{\mathit{ex}}(\Ho(L_B\pref{B}),\Ho(L_A\pref{A}))$$
(where $\sHom_!^{\mathit{ex}}(\Ho(L_B\pref{B}),\Ho(L_A\pref{A}))$
denotes the category of cocontinuous morphisms of derivators
which preserve finite homotopy limits).
This is a $2$-functor from $\tau^{\leq 2}L(\cat)$ to the category of groupoids
which is corepresentable precisely by $A$. This can be reformulated
by saying that we can reconstruct the homotopy type of $A$
from the ``topos'' $\Ho(L_A\pref{A})$.
This is the derivator version of To\"en's homotopy Galois
theory~\cite{to1}.
\end{paragr}

\begin{paragr}
Corollary \ref{catbrowncor} can also be used
to understand the compatibilities of
the equivalences of categories of type \eqref{galois3}
with composition of morphisms of derivators.
More precisely, we have a bicategory
$\ho(\mathcal{Dist})$, whose objects
are the small categories, and whose category
of morphisms from $A$ to $B$
is the homotopy category $\ho(\cat/A\times B)$
(composition is defined by homotopy
fiber products). We will finish this section by explaining
how Corollary \ref{catbrowncor} implies that the functors
\eqref{galois3} define a bifunctor from $\ho(\mathcal{Dist})$
to the $2$-category of derivators.
Define a bicategory $\mathit{SP}$
as follows. The obects of $\mathit{SP}$ are the small categories.
Given two small categories $A$ and $B$, the category
of morphisms $\mathit{SP}(A,B)$ is the full subcategory
of $\cat/A\times B$ whose objects are the
functors $(p,q):C\To A\times B$
such that $p$ and $q$ are smooth and proper.
The composition law of $\mathit{SP}$ is defined
by fiber products (which is meaningful,
as the smooth and proper functors are stable
by pullbacks and compositions).

We denote by $\mathcal{Der}_!$ the $2$-category whose
objects are the derivators, and whose morphisms
are the cocontinuous morphims ($2$-cells are just
$2$-cells in the $2$-category of prederivators).
\end{paragr}

\begin{lemma}\label{lemmadistder}
The functors \eqref{galois5} define a bifunctor
$$s:\op{\mathit{SP}}\To\mathcal{Der}_!$$
\end{lemma}

\begin{proof}
Consider a commutative diagram
$$\xymatrix{
&&G\ar[dl]_t\ar[dr]^u&&\\
&E\ar[dl]_p\ar[dr]^q&&F\ar[dl]_r\ar[dr]^s&\\
A&&B&&C}$$
in which the square is a pullback, and all the maps
are smooth and proper. Note that for any smooth
and proper map $\varphi$, both
$\varphi$ and $\op{\varphi}$
are locally constant; see \cite[Corollary 6.4.8]{Ci3}.
By virtue of Propositions \ref{serre} and \ref{functquillenlc2},
we can apply \cite[Proposition 3.2.28]{Mal} to get that the
base change map
$$\derL u_! \derR t^*\To \derR r^*\derL q_!$$
is an isomorphism in $\sHom_!(\Ho(L_E\pref{E}),\Ho(L_F\pref{F}))$.
In particular, we get a canonical isomorphism
$$\derL s_!\derL u_! \derR t^*\derR p^*\simeq \derL s_!\derR r^*\derL q_!\derR p^*\, .$$
These isomorphisms, together with the functors \eqref{galois5}
define a bifunctor: to check the coherences, we are reduce
to check that commutative squares with the Beck-Chevalley
property are stable by compositions, which is
well known to hold.
\end{proof}

\begin{paragr}
We define now a bicategory $\ho(\mathit{SP})$
as follows. The objects are the small categories, and
given two objects $A$ and $B$, the category of morphisms
from $A$ to $B$ is $\ho(\mathit{SP}(A,B))$, that is the
localisation of $\mathit{SP}(A,B)$ by the class of weak
equivalences. We have a localization
bifunctor
\begin{equation}
\gamma:\mathit{SP}\To\ho(\mathit{SP})\, .
\end{equation}
Corollary \ref{catbrowncor} can now be reformulated:
the canonical bifunctor
\begin{equation}\label{jdisteq}
j:\ho(\mathit{SP})\To\ho(\mathcal{Dist})
\end{equation}
is a biequivalence.
\end{paragr}

\begin{prop}
The equivalences of categories \eqref{galois3}
define a bifunctor
$$S:\op{\ho(\mathcal{Dist})}\To\mathcal{Der}_!\, .$$
\end{prop}

\begin{proof}
For any small categories $A$, $B$ and $C$, we have
$$\ho(\mathit{SP}(A,B)\times\mathit{SP}(A,B))=\ho(\mathit{SP}(A,B))\times\ho(\mathit{SP}(B,C))\, .$$
The universal property of localizations and Lemma \ref{lemmadistder}
imply that we get a bifunctor
$$S':\op{\ho(\mathit{SP})}\To\mathcal{Der}_!\, .$$
We deduce from this and from the biequivalence \eqref{jdisteq}
that there is a unique way to define a bifunctor $S$
from $\op{\ho(\mathcal{Dist})}$ to $\mathcal{Der}_!$
from the equivalences of categories \eqref{galois3}
such that $S'j=S$.
\end{proof}

\begin{sch}
Theorem \ref{univ2} asserts that for any
derivator $\der$, we have an equivalence of derivators
$$\derhom_!(\Ho(L_A\pref{A}),\der)\simeq\lcc(A,\der)\, .$$
As $\derhom_!(-,\der)$ is a $2$-functor, we deduce from
the preceding proposition that we get a bifunctor
$$\lcc(-,\der):\ho(\mathcal{Dist})\To\mathcal{Der}_!\, .$$
which sends a small category $A$ to $\lcc(A,\der)$.
This defines a bifunctor
$$\lcc(-,-):\ho(\mathcal{Dist})\times \mathcal{Der}_!\To\mathcal{Der}_!,$$
which defines an enrichment of $\mathcal{Der}_!$ in
homotopy distributors.
\end{sch}

\bibliography{bibdend}
\bibliographystyle{amsalpha}
\end{document}